\documentclass[10pt,reqno]{amsart}

\newtheorem{thm}{Theorem}
\newtheorem{lem}{Lemma}
\newtheorem{cor}{Corollary}

\usepackage{lipsum}
\usepackage{color}
\usepackage{faktor}

\newcommand{\R}{\mathbb{R}}
\newcommand{\Z}{\mathbb{Z}}

\newcommand{\I}{\mathcal{I}}
\newcommand{\J}{\mathcal{J}}
\newcommand{\PP}{\mathcal{P}}
\newcommand{\A}{\mathcal{A}}
\newcommand{\G}{\mathcal{G}}

\DeclareMathOperator{\spn}{span}
\DeclareMathOperator{\init}{in}
\def\newop#1{\expandafter\def\csname #1\endcsname{\mathop{\rm
#1}\nolimits}}

\newop{Ehr}
\newop{conv}
\newop{cone}
\newop{asc}
\newop{lcm}
\newop{red}
\newop{mult}
\newop{lex}
\newop{sort}
\newop{Supp}

\begin{document}

\title[Properties of Lecture Hall Polytopes]{Some Algebraic Properties of Lecture Hall Polytopes}

\author{Petter Br\"and\'en}
\date{\today}
\address{Matematik, KTH, SE-100 44 Stockholm, Sweden}
\email{pbranden@kth.se}

\author{Liam Solus}
\date{\today}
\address{Matematik, KTH, SE-100 44 Stockholm, Sweden}
\email{solus@kth.se}





\begin{abstract}
In this note, we investigate some of the fundamental algebraic and geometric properties of $s$-lecture hall simplices and their generalizations.  
We show that all $s$-lecture hall order polytopes, which simultaneously generalize $s$-lecture hall simplices and order polytopes, satisfy a property which implies the integer decomposition property.  
This answers one conjecture of Hibi, Olsen and Tsuchiya.  
By relating $s$-lecture hall polytopes to alcoved polytopes, we then use this property to show that families of $s$-lecture hall simplices admit a quadratic Gr\"obner basis with a square-free initial ideal.  
Consequently, we find that all $s$-lecture hall simplices for which the first order difference sequence of $s$ is a $0,1$-sequence have a regular and unimodular triangulation.  
This answers a second conjecture of Hibi, Olsen and Tsuchiya, and it gives a partial answer to a conjecture of Beck, Braun, K\"oppe, Savage and Zafeirakopoulos.
\end{abstract}




\maketitle
\thispagestyle{empty}

\section{Introduction}
\label{sec: introduction}
Let $s = (s_1,\ldots,s_n)$ be a sequence of positive integers.  
An \emph{$s$-lecture hall partition} is a (lattice) point in $\Z^n$ living in the \emph{$s$-lecture hall cone} 
\[
C_n^s := \left\{\lambda\in\R^n : 0\leq\frac{\lambda_1}{s_1}\leq\frac{\lambda_2}{s_2}\leq\cdots\leq\frac{\lambda_n}{s_n}\right\}.
\]
The $s$-lecture hall partitions are generalizations of the \emph{lecture hall partitions}, which correspond to the special case where $s = (1,2,\ldots,n)$.
Lecture hall partitions were first studied by Bousquet-M\'{e}lon and Eriksson \cite{BME97} who proved that
\[
\sum_{\lambda\in C_n^{(1,2,\ldots,n)}\cap\Z^n}q^{\lambda_1+\cdots+\lambda_n} = \frac{1}{\prod_{i=1}^n 1-q^{2i-1}}.
\]
In \cite{SS12}, the \emph{$s$-lecture hall simplex} is defined to be the lattice polytope
\[
P_n^s :=\{\lambda\in C_n^s : \lambda_n \leq s_n \}.
\]
A $d$-dimensional lattice polytope $P\subset\R^n$ is the convex hull of finitely many points in $\Z^n$ whose affine span is dimension $d$.  
For a positive integer $k$, we define $kP :=\{kp\in\R^n: p\in P\}$.  
The generating function
\[
1+\sum_{k>0}|kP\cap\Z^n|x^k = \frac{h_0^\ast+h_1^\ast x+\cdots+h_d^\ast x^d}{(1-x)^{d+1}},
\]
is called the \emph{Ehrhart series} of $P$, and the polynomial $h_0^\ast+h_1^\ast x+\cdots+h_d^\ast x^d$ is called the \emph{(Ehrhart) $h^\ast$-polynomial} of $P$.  
The $h^\ast$-polynomial has only nonnegative integer coefficients, and for the $s$-lecture hall simplex $P_n^s$ it is called the \emph{$s$-Eulerian polynomial}.  
In the case that $s = (1,2,\ldots,n)$, the $s$-Eulerian polynomial is the classic $n^{th}$ Eulerian polynomial, which enumerates the permutations of $[n]$ via the descent statistic.  
One remarkable feature of this generalization is that every $s$-Eulerian polynomial has only real zeros, and thus they each have a log-concave and unimodal sequence of coefficients \cite{SV15}.  
Identifying large families of lattice polytopes with unimodal $h^\ast$-polynomials is a popular research topic with natural connections to the algebra and geometry of the toric varieties associated to lattice polytopes.
Showing that an $h^\ast$-polynomial is real-rooted is a common approach to proving unimodality results in geometric and algebraic combinatorics \cite{BJM19,GS18,SV15,S19}.  
However, the applicability of this proof technique to families of $h^\ast$-polynomials does not obviously relate to the algebraic structure of the associated toric variety for the underlying polytopes.  
Consequently, research into the algebraic properties of the $s$-lecture hall simplices and their generalizations that can be used to verify unimodality of the associated $h^\ast$-polynomials is an ongoing and popular topic \cite{BBKSZ15a,BBKSZ15b,BL16,HOT16,PS13a,PS13b,SS12,SV15}.  

In this note, we prove some fundamental algebraic properties of $s$-lecture hall simplices and their generalizations.  
We show that all \emph{$s$-lecture hall order polytopes} \cite{BL16}, a common generalization of $s$-lecture hall simplices and order polytopes, have the integer decomposition property.  
This result positively answers a conjecture of \cite{HOT16}.  
As an application of this result, we then give an explicit description of a quadratic and square-free Gr\"obner basis for the affine toric ideal of families of $s$-lecture hall simplices.  
To do so, we relate $s$-lecture hall polytopes to \emph{alcoved polytopes} \cite{LP07}.
The identified Gr\"obner basis is purely lexicographic and can be constructed for any toric ideal associated to an $s$-lecture hall simplex for which the first-order difference sequence of $s$ is a $0,1$-sequence.  
This answers a second conjecture of \cite{HOT16} in a special case that they noted to be of particular interest, and it provides a partial answer to the conjecture of \cite{BBKSZ15b}.  


\section{The Integer Decomposition Property for $s$-Lecture Hall Order Polytopes}
\label{sec: the integer decomposition property}

\subsection{The algebraic structure of a lattice polytope}
\label{subsec: the algebraic structure of a lattice polytope}
There are two important algebraic objects associated to a lattice polytope $P\subset\R^n$.  
The first is its \emph{toric ideal}, the zero locus of which is the \emph{affine toric variety of $P$}.  
The second is the \emph{Ehrhart ring of $P$}, which is a graded and semistandard semigroup algebra associated to $P$.  
The integer decomposition property is precisely the property that tells us when the coordinate ring of the affine toric variety of $P$ coincides with its Ehrhart ring.
Hence, it is desirable to know if a family of lattice polytopes admits this property.  

For a lattice polytope $P\subset\R^n$, define the \emph{cone over $P$} to be the convex cone
\[
\cone(P) :=\spn_{\R^{n+1}}\{(p,1)\in\R^n\times\R : p\in P\}\subset\R^{n+1}.
\]
To any integer point $z = (z_1,\ldots,z_{n+1})\in\Z^{n+1}$ we associate a laurent monomial $t^z:=t_1^{z_1}t_2^{z_2}\cdots t_{n+1}^{z_{n+1}}$.  
Let $\{a_1,\ldots,a_m\} := \{(p,1)\in\R^n\times \R : p\in P\cap\Z^n\}$, and let $K[{\bf x}] :=K[x_1,\ldots,x_m]$ denote the polynomial ring over the field $K$ in $m$ indeterminates.
The \emph{toric ideal} of $P$, denoted $\I_P$, is the kernel of the semigroup algebra homomorphism
\begin{equation*}
\Phi : K[{\bf x}]\longrightarrow K[t_1,t_2,\ldots,t_{n+1},t_1^{-1},t_2^{-1},\ldots,t_{n+1}^{-1}] 
\qquad 
\mbox{where}
\qquad
\Phi : x_i\mapsto t^{a_i}.
\end{equation*}

For $k\in\Z_{>0}$ we let $kP:=\{kp:p\in P\}$ denote the $k^{th}$ dilate of $P$, and we let $A(P)_k$ denote the vector space (over $K$) spanned by the monomials $t_1^{z_1}t_2^{z_2}\cdots t_n^{z_n}t_{n+1}^k$ for $z \in kP\cap\Z^n$.  
Since $P$ is convex we have that $A(P)_kA(P)_r\subset A(P)_{k+r}$ for all $k,r\in\Z_{>0}$.
It follows that the graded algebra 
\[
A(P) :=\bigoplus_{k=0}^\infty A(P)_k
\]
is finitely generated over $A(P)_0 := K$, and we call it the \emph{Ehrhart Ring of $P$}.  
Equivalently, $A(P)$ is the semigroup algebra $K[t^z: z\in\cone(P)\cap\Z^{n+1}]$ with the grading $\deg(t_1^{z_1}\cdots t_{n+1}^{z_{n+1}}) = z_{n+1}$.  
\begin{def}
\label{def: IDP}
A polytope $P\subset\R^n$ has the \emph{integer decomposition property}, or is IDP (or is \emph{integrally closed}), if for every positive integer $k$ and every $z\in kP\cap \Z^n$, there exist $k$ points $z^{(1)},z^{(2)},\ldots,z^{(k)}$ such that
$z = \sum_iz^{(i)}$.
\end{def}
Since the coordinate ring of the toric ideal $\I_P$ is 
$
K[{\bf x}]/\I_P\cong K[t^{a_1},\ldots,t^{a_m}],
$
the polytope $P$ is IDP if and only if this quotient ring is isomorphic to $A(P)$.  
In this case, the toric algebra of $\I_P$ can be used to recover the Ehrhart theoretical data encoded in $A(P)$.  
Therefore, it is desirable to understand when combinatorially interesting polytopes are IDP.

\subsection{$s$-Lecture hall order polytopes}
\label{subsec: s-lecture hall order polytopes}
A \emph{labeled poset} is a partially ordered set $\PP$ on $[n]:=\{1,2,\ldots,n\}$ for some positive integer $n$; that is, $\PP = ([n], \preceq)$ where $\preceq$ denotes the partial order imposed on the ground set $[n]$.  
In the following, we let $\leq$ denote the usual total order on the integers.  
We say that $\PP$ is \emph{naturally labeled} if it is a labeled poset for which $i\leq j$ whenever $i\preceq j$.  
Let $s = (s_1,\ldots,s_n)$ be a sequence of positive integers and let $\PP = ([n],\preceq)$ be a naturally labeled poset.  
A \emph{$(\PP,s)$-partition} is a map $\lambda: [n]\longrightarrow \R$ such that
\[
\frac{\lambda_i}{s_i}\leq\frac{\lambda_j}{s_j}
\qquad 
\mbox{whenever}
\qquad
i\prec j,
\]
where we let $\lambda_i$ denote $\lambda(i)$ for all $i\in[n]$.
The \emph{$s$-lecture hall order polytope} associated to $(\PP,s)$ is the convex polytope
\[
O(\PP,s) := \left\{\lambda\in\R^n : \mbox{ $\lambda$ is a $(\PP,s)$-partition and $0\leq\lambda_i\leq s_i$ for all $i\in[n]$}\right\}.
\]
The $s$-lecture hall order polytopes were introduced in \cite{BL16} as a common generalization of the well-known order polytopes and the $s$-lecture hall simplices.
When $s = (1,1,\ldots, 1)$, then $O(\PP,s)$ is the order polytope associated to $\PP$, and when $\PP$ is the $n$-chain $O(\PP,s) = P_n^s$.  
In \cite{HOT16}, it is conjectured that all $s$-lecture hall simplices are IDP.  
We now prove a more general (and stronger) statement.  

A poset $\PP = ([n],\preceq_\PP)$ is called a \emph{lattice} if every pair of elements $a,b\in[n]$ have both a least upper bound and a greatest lower bound in $\PP$.  
An element $c\in[n]$ is a \emph{least upper bound} of $a$ and $b$ in $\PP$ if $a\preceq_\PP c$, $b\preceq_\PP c$ and whenever $d\in[n]$ satisfies $a\preceq_\PP d$ and $b \preceq_\PP d$ then $c\preceq_\PP d$.
Analogously, $c\in[n]$ is a \emph{greatest lower bound} of $a$ and $b$ in $\PP$ if $a\succeq_\PP c$, $b\succeq_\PP c$ and whenever $d\in[n]$ satisfies $a\succeq_\PP d$ and $b \succeq_\PP d$ then $c\succeq_\PP d$.
Whenever a least upper bound or greatest lower bound exists, it is unique.  
So we let $a\vee b$ denote the least upper bound of $a$ and $b$ in $\PP$ and $a\wedge b$ denote their greatest lower bound.
A lattice $\PP$ is called \emph{distributive} if for all triples of elements $a,b,c$ in $\PP$ we have that
\[
a\vee(b\wedge c)  = (a\vee b) \wedge (a\vee c).
\]
Let $\Lambda(\PP,s)$ denote the collection of all maps $\lambda:[n]\longrightarrow\Z$ satisfying
\[
\frac{\lambda_i}{s_i}\leq\frac{\lambda_j}{s_j}
\qquad 
\mbox{whenever}
\qquad
i\preceq j.
\]
In general, we will identify a map $p: [n]\longrightarrow\R$ with the point $(p_1,\ldots,p_n)\in\R^n$.  
Conversely, every point $p\in\R^n$ corresponds to a map $p:[n]\longrightarrow\R$.  
Note that $\Lambda(\PP,s)$ is a distributive sublattice of $\Z^n$, under the usual product ordering. 
Moreover, $\Lambda(\PP,s)= \Lambda(\PP,s) + \Z (s_1,\ldots,s_n)$ and 
$kO(P,s)\cap \Z^n= \Lambda(\PP,s) \cap \prod_{i\in[n]}[0,ks_i]$ for all $k \in \Z_{>0}$. 
Let $\lambda,\gamma \in O(P,s)\cap \Z^n$. We write $\lambda \unlhd \gamma$ provided that 
\begin{enumerate}
\item $\lambda_i \leq \gamma_i$ for all $i \in [n]$,
\item if $\lambda_i \neq 0$, then $\gamma_i = s_i$. 
\end{enumerate}

\begin{thm}
\label{thm: IDP}
Let $\PP = ([n],\preceq)$ be a naturally labeled poset and let $s = (s_1,\ldots,s_n)$ be a sequence of positive integers.
If $\lambda \in kO(\PP,s)\cap \Z^n$ for $k\in\Z_{>0}$, then there are unique elements $\lambda^{(1)}, \ldots, \lambda^{(k)} \in O(\PP,s)\cap \Z^n$ such that 
\begin{equation}\label{nicechain}
\lambda^{(1)} \unlhd \lambda^{(2)} \unlhd \cdots \unlhd \lambda^{(k)}  \ \ \mbox{ and } \ \  \lambda= \lambda^{(1)}+\lambda^{(2)}+\cdots+\lambda^{(k)}.
\end{equation}
Moreover, if 
\[
\lambda= \gamma^{(1)} +\gamma^{(2)}+\cdots+\gamma^{(m)}
\]
where $m \leq k$ and $\gamma^{(1)}, \ldots, \gamma^{(k)} \in O(P,s)\cap \Z^n$, then $\lambda^{(1)}(x) \leq \gamma^{(i)}(x) \leq \lambda^{(k)}(x)$ for all $x \in [n]$ and $i\in[k]$.  
\end{thm}

\begin{proof}
We first prove the existence of \eqref{nicechain} by induction over $k\geq 1$. 
Suppose $\lambda \in kO(\PP,s)\cap \Z^n$, where $k>1$, and write $\lambda=\lambda\wedge s+(\lambda-s)\vee 0$. 
Then $\lambda \wedge s \in O(\PP,s)\cap \Z^n$ and $(\lambda-s)\vee 0 \in (k-1)O(\PP,s)\cap \Z^n$. Let $\lambda^{(k)}=\lambda\wedge s$. 
By induction we may write 
\[
(\lambda-s)\vee 0=\lambda^{(1)} + \cdots + \lambda^{(k-1)}
\]
where $\lambda^{(1)},\ldots, \lambda^{(k-1)}$ satisfies \eqref{nicechain}. 
Clearly $\lambda^{(i)} \leq \lambda^{(k)}$ for all $1 \leq i \leq k-1$.  
Moreover if $\lambda^{(i)}(x) \neq 0$ for some $1 \leq i \leq k-1$, then $((\lambda-s)\vee 0)(x) \neq 0$, and thus $\lambda^{(k)}(x)= (\lambda\wedge s)(x)=s(x)$ as desired. 
This establishes \eqref{nicechain}.

Suppose now that the sequence $\lambda^{(1)},\ldots, \lambda^{(k)}$ satisfies \eqref{nicechain}.  
Note $\lambda(x)>s(x)$ if and only if $\lambda^{(i)}(x) >0$ for at least two distinct $i$, and this happens if and only if $\lambda^{(k-1)}(x)>0$ and $\lambda^{(k)}(x) =s(x)$. 
Hence, $\lambda^{(k)}= \lambda\wedge s$. 
The uniqueness of $\lambda^{(1)}, \ldots, \lambda^{(k)}$ then follows by induction.  

Suppose next that
\[
\lambda= \gamma^{(1)} +\gamma^{(2)}+\cdots+\gamma^{(m)} \in kO(\PP,s)\cap \Z^n,
\]
where $m \leq k$ and $\gamma^{(1)}, \ldots, \gamma^{(m)} \in O(\PP,s)\cap \Z^n$. 
Then $\gamma^{(i)}(x) \leq \min\{\lambda(x),s(x)\}=\lambda^{(k)}(x)$. 
If $\gamma^{(i)}(x) < \lambda^{(1)}(x)$ for some $x \in [n]$ and $1 \leq i \leq m$, then 
$\lambda^{(i)}(x)=s(x)$ for all $2\leq i \leq k$ (since $\lambda^{(1)}(x) \neq 0$). 
Hence, $\lambda^{(1)}(x)= \lambda(x)-(k-1)s(x)>0$ and 
\[
\lambda(x)-\gamma^{(i)}(x) =\lambda^{(1)}(x)-\gamma^{(i)}(x)+(k-1)s(x) > (k-1)s(x),
\]
which is a contradiction since $\lambda-\gamma^{(i)} \in (m-1)O(\PP,s)\cap \Z^n$. 
\end{proof}

It follows from Theorem~\ref{thm: IDP} that all $s$-lecture hall order polytopes are IDP.  
In the remainder of this note, we use Theorem~\ref{thm: IDP} to identify a regular and unimodular triangulation of some $s$-lecture hall polytopes by computing a quadratic and square-free Gr\"obner basis for their associated toric ideals.

\section{A Quadratic and Square-Free Gr\"obner Basis for Some $s$-Lecture Hall Simplices}
\label{sec: quadratic and square-free grobner bases}
Let $P\subset\R^n$ be a lattice polytope and let $\A := \{(p,1)\in\R^n\times \R : p\in P\cap\Z^n\}$.  
Label $\A$ as $\A = \{a_1,\ldots,a_m\}$, and suppose that $\succ$ is a term order on the polynomial ring $K[{\bf x}]:=K[x_1,\ldots,x_m]$; that is, $\succ$ is a total order on the monomials in $K[{\bf x}]$ satisfying
\begin{enumerate}
\item $x^a\succ x^b$ implies that $x^ax^c\succ x^bx^c$ for all $c\in\Z^n_{\geq0}$, and
\item $x^a\succ x^0 = 1$ for all $a\in\Z^n_{>0}$.  
\end{enumerate}  
%
Given a polynomial $f = \sum_{a\in\Z^n_{\geq0}}c_ax^a$ with coefficients $c_a\in K$ we call the set 
\[
\Supp(f) :=\{a\in\Z^n : c_a\neq0\}
\]
the \emph{support} of $f$.  
Fixing a term order $\succ$ on the monomials in $K[{\bf x}]$, we define the \emph{initial term} of $f$ to be the term $c_ax^a$ for which $x^a\succ x^b$ for every $b\in\Supp(f)\setminus\{a\}$.
We denote the initial term of $f$ with respect to the term order $\succ$ by $\init_\succ(f)$.
Given an ideal $\I\subset K[{\bf x}]$, the \emph{initial ideal of $\I$ with respect to $\succ$} is 
\[
\init_\succ(\I) := \langle \init_\succ(f) : f\in\I\rangle.
\]
A finite set of polynomials $\G_\succ(\I) :=\{g_1,\ldots,g_p\}$ is called a \emph{Gr\"obner basis} of $\I$ with respect to $\succ$ if $\init_\succ(\I) = \langle\init_\succ(g_1),\ldots,\init_\succ(g_p)\rangle$.  
If $\left\{\init_\succ(g_1),\ldots,\init_\succ(g_p)\right\}$ is the unique minimal generating set for $\init_\succ(\I)$, then $\G_\succ(\I)$ is called \emph{minimal}.  
A minimal Gr\"obner basis $\G_\succ(\I)$ if further called \emph{reduced} if no non-inital term of any $g_i$ is divisible by some element of $\left\{\init_\succ(g_1),\ldots,\init_\succ(g_p)\right\}$.  
The monomials of $K[{\bf x}]$ that are not in $\init_\succ(\I)$ are called the \emph{standard monomials} of $\init_\succ(\I)$.

Let $P\subset\R^n$ be a lattice polytope and let $\A := \{(p,1)\in\R^n\times \R : p\in P\cap\Z^n\}$.  
We denote the sublattice of $\Z^{n+1}$ spanned by the lattice points in $\A$ by $\Z\A$.  
Any sufficiently generic height function $\omega : \A\longrightarrow \R_{\geq0}$ on the points in $\A$ induces a term order $\succ_\omega$ on $K[{\bf x}]$ and yields a corresponding Gr\"obner basis $\G_{\succ_\omega}(\I_P)$ for the toric ideal $\I_P$ of $P$.  
On the other hand, the collection of faces of 
\[
\conv\{(a_i,\omega(a_i))\in\R^{n+1}: i\in[m]\}
\]
that minimize some linear functional in $\R^{n+1}$ with a negative $(n+1)^{st}$ coordinate correspond to the faces of a \emph{regular triangulation} $\Delta_\omega$ of $P$ given by projecting these faces onto $P$ in $\R^n$.  
The fundamental correspondence between the regular triangulation $\Delta_\omega$ and the Gr\"obner basis $\G_{\succ_\omega}(\I_P)$ states that the square-free standard monomials $x_{i_1}x_{i_2}\cdots x_{i_\ell}$ with respect to $\init_{\succ_\omega}(\I_P)$ correspond to the faces $\conv\{a_{i_1},a_{i_2},\ldots,a_{i_k}\}$ of $\Delta_\omega$ \cite[Theorem 8.3]{S96}.  
Furthermore, if the $\init_{\succ_\omega}(\I_P)$ is \emph{square-free} (i.e. generated by square-free monomials), then the simplices of $\Delta_\omega$ have smallest possible volume (i.e. are \emph{unimodular}) with respect to the lattice $\Z\A$.  
In this case, the regular triangulation $\Delta_\omega$ is called \emph{unimodular}.
If $\init_{\succ_{\omega}}(\I_P)$ consists only of quadratic monomials, then $\Delta_\omega$ is \emph{flag} triangulation, meaning its minimal non-faces are pairs of points $\{a_i,a_j\}$.  
When an $n$-dimensional lattice polytope $P$ is IDP, then $\Z\A = \Z^{n+1}$, and a square-free Gr\"obner basis for $\I_P$ identifies a regular unimodular triangulation of $P$ with respect to the lattice $\Z^n$.

Our goal in this section is to identify a quadratic Gr\"obner basis with a square-free initial ideal for the toric ideals of a subcollection of $s$-lecture hall simplices that includes the lecture hall simplex $P_n^{(1,2,\ldots,n)}$.  
This is the first explicit description of such a Gr\"obner basis for the toric ideal of $P_n^{(1,2,\ldots,n)}$.  
In the remainder of this section, we will assume that $s = (s_1,\ldots,s_n)$ is a weakly increasing sequence of positive integers satisfying $0\leq s_i-s_{i-1}\leq 1$ for all $i\in[n]$; that is, we will assume that the first-order difference sequence of $s$ is a $0,1$-sequence.

\subsection{$s$-lecture hall simplices and alcoved polytopes}
\label{subsec: lecture hall multisets}
To produce the desired quadratic and square-free Gr\"obner basis for the toric ideal of $P_n^s$ we will use the following transformation.  
For the sequence $s = (s_1,\ldots,s_n)$, set $s_{n+1} := s_n+1$.  
Notice that since $s$ is assumed to be weakly increasing then $x_i\leq x_{i+1}$ for all $i\in[n]$ and $x\in P_n^s$. 
Now consider the unimodular transformation 
\[
\varphi: \R^n\longrightarrow \R^n; \qquad \varphi: x_i\mapsto x_i-x_{i-1},
\quad 
\mbox{where }
 x_0:=0,
\]
and the homogenizing affine transformation
\[
h: \R^n\longrightarrow \R^{n+1}; \qquad h: x \mapsto \left(x_1,\ldots,x_n,s_{n+1}-\sum_{i=1}^nx_i\right).
\]
Then the convex lattice polytope $A_n^s:=(h\circ\varphi)(P_n^s)$ is defined by the linear inequalities 
\[
\begin{split}
0\leq x_1+\cdots+x_j&\leq s_j, \mbox{ for all $j\in[n]$},\\
0\leq (s_{j+1}-s_j)(x_1+\cdots+x_j)&\leq s_jx_{j+1}, \mbox{ for all $j\in[n-1]$, and}\\
x_1+\cdots+x_{n+1}&=s_{n+1}.\\
\end{split}
\]
The following lemma notes that the lattice points within $A_n^s$ consist of the lattice points in the \emph{alcoved polytope} \cite{LP07} defined by the inequalities
\[
\begin{split}
0\leq x_1+\cdots+x_j&\leq s_j, \mbox{ for all $j\in[n]$, and }\\
x_1+\cdots+x_{n+1}&=s_{n+1}.\\
\end{split}
\]
that satisfy a useful combinatorial condition.  
Conditions $(1)$ and $(2)$ of the lemma specify that a lattice point in $A_n^s$ must lie in this alcoved polytope, and conditions $(3)$ and $(4)$ constitute the combinatorial criterion we desire.  
\begin{lem}
\label{lem: generalizing the lattice point interpretation}
Suppose that $s$ is a weakly increasing sequence of positive integers for which the first-order difference sequence is a $0,1$-sequence.  
Than a lattice point $(z_1,\ldots,z_{n+1})$ is in $A_n^s\cap\Z^{n+1}$ if and only if the following conditions hold:
\begin{enumerate}
	\item $z_1+\cdots+z_{n+1}=s_{n+1}$,
	\item $0\leq z_1+\cdots+z_j\leq s_j$ for all $j\in[n+1]$, 
	\item whenever $s_{i+1}-s_i = 0$ then $0\leq z_{i+1}$, and
	\item whenever $s_{i+1}-s_i\neq 0$ and $z_{k}\neq 0$ for some $k<i+1$, then $z_{i+1}\neq0$.
\end{enumerate}
\end{lem}

\begin{proof}
Suppose first that $(z_1,\ldots,z_{n+1})\in A_n^s\cap\Z^{n+1}$.  
Then certainly conditions (1) and (2) hold.  
To see that conditions (3) holds, suppose that $s_i-s_{i+1} = 0$.  
Then by the defining inequalities for $A_n^s$, we know that $0\leq z_{i+1}$.  
Finally, to see condition (4) holds, suppose that $s_{i+1}-s_i\neq 0$ and that $z_{k}\neq0$ for some $k<i+1$.  
Then since $s_{i+1}-s_i\neq 0$, the inequality
\[
0\leq (s_{j+1}-s_j)(z_1+\cdots+z_k+\cdots+z_j)\leq s_jz_{j+1}
\]
reduces to
\[
0\leq z_1+\cdots+z_k+\cdots+z_j\leq s_jz_{j+1},
\]
and since $z_k\neq 0$, it follows that $z_{i+1}\neq 0$.  

Conversely, suppose that $(z_1,\ldots,z_{n+1})$ is a lattice point satisfying the conditions (1), (2), (3), and (4).  
Then by conditions (1) and (2) it suffices to show that $(z_1,\ldots,z_{n+1})$ satisfies the inequalities
\[
0\leq (s_{j+1}-s_j)(z_1+\cdots+z_j)\leq s_jz_{j+1}
\]
for all $i\in[n]$.  
However, since $(z_1,\ldots,z_{n+1})$ is a lattice point, whenever $z_{i+1}\neq 0$, we know that $z_{i+1}\geq 1$.  
Thus, the conditions (3) and (4) show that $s_{i+1}-s_i\leq z_{i+1}$ for all $i\in[n]$.  
Therefore, condition (2) implies that the inequalities 
\[
0\leq (s_{j+1}-s_j)(z_1+\cdots+z_j)\leq s_jz_{j+1}
\]
hold for all $i\in[n]$.  
\end{proof}

Let $A(s):=\{i\in[n+1] : s_{i-1}<s_i\}$ denote the collection of indices $i\in[n+1]$ for which $s_{i}-s_{i-1}\neq 0$.  
Notice that a lattice point $z = (z_1,\ldots,z_{n+1}) \in A_n^s\cap\Z^{n+1}$ indexes the multiset
\[
\{1^{z_1},2^{z_2},\ldots,(n+1)^{z_{n+1}}\}.
\]
We call any such multiset an \emph{$s$-lecture hall multiset (of order $n$)}.    
The notion of multisets and their corresponding lattice points in $\Z^n$ will be used in the coming sections.  
It will be useful to have the following notation.  

For $i\in[n]$ and a multiset $I$ of $[n]$ we let $\mult_I(i)$ denote the multiplicity of $i$ in $I$.  
Given a collection of multisets $\I = \{I_1,\ldots, I_k\}$, we let $\Sigma\I$ denote the multiunion $\bigcup_{I\in\I}I$.  
For each multiset $I_i\in\I$ we let $x^i:=(\mult_{I_i}(1),\mult_{I_i}(2),\ldots,\mult_{I_i}(n))\in\Z^n$ denote its \emph{multiplicity vector}.  
The multiplicity vectors $x^1,\ldots,x^k$ can be ordered \emph{lexicographically}, i.e., for two vectors $x,y\in\Z^n$ we say $x\succ_{\lex}y$ if and only if the leftmost nonzero entry in $x-y$ is positive.  
Given this, we may reindex the collection $\I$ such that $x^1\succ_{\lex} x^2 \succ_{\lex}\cdots \succ_{\lex} x^k$.  
Moreover, the lexicographic ordering on the lattice points in $\Z^n$ induces a lexicographic ordering on the multisets of $[n]$.  That is, for two multisets $I_1,I_2$ of $[n]$, we say $I_1\succ_{\lex} I_2$ if and only if $x^1\succ_{\lex} x^2$.  
Furthermore, given two collections of $k$ multisets $\I=  \{I_1,\ldots, I_k\}$ and $\J  = \{J_1,\ldots,J_k\}$ of $[n]$, we write $\I \succ \J$ if and only if the $I_k\succ_{\lex} J_k$ for the smallest index $k$ for which $I_k\neq J_k$.  
A collection $\I=  \{I_1,\ldots, I_k\}$ of $k$ multisets is said to be \emph{minimal} if $\I^\prime \succ_{\lex} \I$ for any collection $\I^\prime$ of $k$ multisets of $[n]$ satisfying $\Sigma\I^\prime = \Sigma\I$.  
Equivalently, the collection of vectors $\{x^1,\ldots,x^k\}$ is called minimal.

\subsection{A lexicographic Gr\"obner basis}
\label{subsec: a lexicographic grobner basis}
We now use the notion of $s$-lecture hall multisets described in Subsection~\ref{subsec: lecture hall multisets} to describe a quadratic Gr\"obner basis with a square-free initial ideal for the toric ideal associated to $P_n^s$. 
In this section, it is unnecessary to speak directly about $s$-lecture hall multisets, but instead, we need only their corresponding multiplicity vectors in $A_n^s\cap\Z^n$.  
If $x \in kA_n^s \cap \Z^{n+1}$, let 
$$
\alpha_r(x) = \min\{i \in [n+1]: x_i \geq r, \mbox{ and } x_j \geq r \mbox{ for all } j>i \mbox{ such that } j \in A(s)\}.
$$
If $x = (x_1,\ldots, x_{n+1})\neq 0$, let $\ell(x)= \min\{ i : x_i \neq 0\}$. 

\begin{lem}
\label{sp}
Suppose $x^1\succeq_{\lex} x^2 \succeq_{\lex} \cdots \succeq_{\lex} x^m$ are integer points in $A_n^s \cap \Z^{n+1}$, and let $y = x^1+ x^2 +\cdots+ x^m$. 
If $x^1\succeq_{\lex} x^2 \succeq_{\lex} \cdots \succeq_{\lex} x^m$ are pairwise minimal, then 
$$
\ell(x^i)= \alpha_i(y), \ \ \ \ \mbox{ for all } 1 \leq i \leq m.
$$
\end{lem}
\begin{proof}
The proof is by induction over $i$, $1\leq i \leq m$. Note that $\alpha_1(y)$ is the first nonzero coordinate of $y$. Since the 
$x^i$ are ordered lexicographically we have $\ell(x^1)= \alpha_1(y)$ as claimed. 

Suppose, now that the claim is true for all indices less than or equal to  $i\geq 1$, but that it is not true for $i+1$. 
Then $a:=\alpha_{i+1}(y)<\ell(x^{i+1})=:b$. 
Let $c$ be the largest integer in $\{j : a\leq j<b, y_j \geq i+1\}$. 
Note that $x_j^k=0$ for all $k \geq i+1$, since the $x^k$ are ordered lexicographically.  
Since $y_c \geq i+1$, there is a $k$ satisfying $1\leq k \leq i$ such that $x^k_c \geq 2$. 
Let $d>c$ be the smallest index for which $x_d>0$ and either $d \notin A(s)$ or $x_d \geq 2$. 
Then the pair $\{x^k+e_d-e_c,x^i +e_c-e_d\}$ is smaller than $\{x^k,x^i\}$, which contradicts pairwise minimality. 
\end{proof}

\begin{thm}
\label{thm: lex minimal}
If $x^1 \succeq_{\lex} x^2 \succeq_{\lex} \cdots \succeq_{\lex} x^k$ are pairwise minimal, then the collection $\{x^1,\ldots,x^k\}$ is minimal. 
\end{thm}

\begin{proof}
Let $y= x^1+ x^2 +\cdots+ x^k$. We prove that $x^1 \succeq_{\lex} x^2 \succeq_{\lex} \cdots \succeq_{\lex} x^k$ are uniquely determined given $y$. The proof is by induction on $k \geq 2$. 

Suppose first that $\alpha_i(y)>\alpha_j(y)$ for some $i<j$. Let $m$ be the last index for which $\alpha_m(y)>\alpha_{m+1}(y)$. Let $u = x^1+\cdots+x^m$ and $v = x^{m+1}+\cdots+x^k$. We prove that $u$ and $v$ are uniquely determined. We claim that if $j \geq \alpha_{m+1}(y)$ and $j \in A(s)$, then 
\begin{equation}\label{yz1}
v_j = \min\left( y_j-m, s_j(k-m)- \sum_{i=1}^{j-1}v_i\right), 
\end{equation}
and if $j \geq \alpha_{m+1}(y)$ and $j \not \in A(s)$, then 
\begin{equation}\label{yz2}
v_j = \min\left( y_j, s_j(k-m)- \sum_{i=1}^{j-1}v_i\right).  
\end{equation}
Assume $j \geq \alpha_{m+1}(y)$ and $j \in A(s)$. 
Then the $j$th coordinate of each $x^i$, $i \leq m$, is positive, since $j \in A(s)$. 
Hence, $v_j=y_j -u_j \leq y_j-m$.  
Moreover, $v_j \leq s_j(k-m)- \sum_{i=1}^{j-1}v_i$ by the defining inequalties of $A_n^s$ and the definition of $s$-lecture hall partitions. 
Thus, if \eqref{yz1} fails, then $v_j<y_j-m$ and $\sum_{i=1}^{j}v_i<s_j(k-m)$. 
So we conclude that there are indices $i,\ell$ such that $i\leq m$ and $\ell\geq m+1$ such that $x^i_j >1$ and $\sum_{i=1}^{j}x^\ell_i<s_j$. 
Let $p>j$ be the smallest index such that $x^\ell_p>1$ or $x^\ell_p=1$ and $p \not \in A(s)$. 
Then the pair $\{x^i-e_j+e_p, x^\ell+e_j-e_p\}$ is smaller than $\{x^i, x^\ell\}$, which contradicts pairwise minimality. 
Thus, \eqref{yz1} follows. 
The case when $j \geq \alpha_{m+1}(y)$ and $j \not \in A(s)$ follows similarly. 

If $\alpha_i(y)=\alpha_j(y)=a$ for all $i,j$, we claim that the first nonzero coordinate of the $x^i$ differ by at most one. 
Indeed if the first nonzero coordinate of $x^i$ and $x^j$, $i<j$, differ by at least two, then let $b$ be the smallest integer greater than $a$ for which the entry in $x^i$ is either greater than one or equal to one and not in $A(s)$. 
Then the pair $\{x^i-e_a+e_b,x^j+e_a-e_b\}$ is smaller than $\{x^i, x^j\}$, which contradicts pairwise minimality. 
If the first nonzero entry of all $x^i$ is equal, we may delete this entry for each $x^i$ and repeat our argument. 
Hence, we reduce to the case when either $\alpha_i(y)=\alpha_j(y)=a$ for all $i,j$ and some first coordinates differ, or $\alpha_i(y)>\alpha_j(y)$. 
The latter case is dealt with above. 
For the former case, let $m$ be the index for which the first coordinates of $x^m$ and $x^{m+1}$ differ. 
Let $u = x^1+\cdots+x^m$ and $v = x^{m+1}+\cdots+x^k$ and argue as above.
\end{proof}

Theorem~\ref{thm: lex minimal} allows us to compute the desired Gr\"obner basis.
Let 
$$
K[{\bf x}]:= K[x_I : \mbox{$I$ is a $s$-lecture hall multiset}]
$$ 
be a polynomial ring over a field $K$ in the indeterminants $x_I$.  
Given a collection of $s$-lecture hall multisets $\mathcal{I} = \{I_1,\ldots,I_r\}$, we denote the monomial $x_{I_1}\cdots x_{I_r}$ by $x^{\mathcal{I}}$.  
In the following, we denote the toric ideal $\I_{A_n^s}$ in $K[{\bf x}]$ simply by $\I_n^s$.  
For a collection of $k$ $s$-lecture hall multisets $\{I_1,\ldots,I_k\}$, we let $\{I_1^-,\ldots,I_k^-\}$ denote the minimal collection of $k$ $s$-lecture hall multisets satisfying $\Sigma\{I_1^-,\ldots,I_k^-\} = \Sigma\{I_1,\ldots,I_k\}$.  
\begin{thm}
\label{thm: lexicographic grobner basis}
There exists a term order $\succ$ on $K[{\bf x}]$ such that the marked set of binomials
$$
G := \{\underline{x_Ix_J} - x_{I^-}x_{J^-} : \mbox{$I$ and $J$ are $s$-lecture hall multisets}\}
$$
is a reduced Gr\"{o}bner basis for $\I_n^s$ with respect to $\succ$.  
The initial ideal $\init_\succ \I_n^s$ is generated by the underlined terms, all of which are square-free.
\end{thm}

\begin{proof}
For a collection of $s$-lecture hall multisets $\mathcal{I} = \{I_1,\ldots, I_r\}$, it is clear that the relation 
$$
x_{I_1}\cdots x_{I_r} - x_{I_1^-}\cdots x_{I_r^-}
$$
lies in the ideal $\I_n^s$.  
This is because the multiunion over each collection of multisets is the same and $I_1^-,\ldots,I_k^-$ are all $s$-lecture hall multisets.  
The binomials in $G$ define a reduction relation on $k[{\bf x}]$ for which the underlined term is treated as the leading term of the binomials.  
We say that a monomial is in \emph{normal} form with respect to a reduction relation if it is the remainder upon division with respect to the given set of polynomials and their specified leading terms \cite[Chapter 3]{S96}.  
It follows from Theorem~\ref{thm: lex minimal} that if $\I$ is not minimal, then there exists some pair $\{I_i,I_j\}\subset\I$ for which $\{I_i,I_j\}$ is not minimal.  
So a monomial $x^{\mathcal{I}}$ for $\mathcal{I} = \{I_1,\ldots, I_r\}$ is in normal form with respect to the reduction relation defined by $G$ if and only if $\mathcal{I}$ is minimal.  
Notice also that the reduction modulo $G$ is Noetherian; i.e., every sequence of reductions modulo $G$ terminates.  
This is because reduction of the monomial $x_{I_1}\cdots x_{I_r}$ by $x_{I_i}x_{I_j}-x_{I_i^-}x_{I_j^-}$ amounts to replacing the multiset $\I =  \{I_1,\ldots, I_r\}$ with the multiset $\I^\prime := \I \backslash\{I_i,I_j\}\cup\{I_i^-,I_j^-\}$.  
Since $\I^\prime$ is lexicographically smaller than $\I$, reduction modulo $G$ is Noetherian.  
So by applying \cite[Theorem 3.12]{S96} we find that $G$ is a coherently marked collection of binomials.  
Thus, it is a Gr\"obner basis for $\I_n^s$ with respect to some term order $\succ$ on $K[{\bf x}]$ and its initial ideal is generated by the underlined terms.  
It follows readily that the monomials in the initial ideal with respect to this term order are precisely the non-minimal monomials.  
Thus, $G$ is a quadratic and reduced Gr\'obner basis for $\I_n^s$ with a square-free initial ideal.
\end{proof}

The following corollary extends the results of \cite{BBKSZ15b} and \cite{HOT16}.  
In particular, it provides a partial answer to \cite[Conjecture 5.2]{HOT16} in a special case that they noted to be of particular interest; namely, when the first order difference sequence of $s$ is a $0,1$-sequence.
\begin{cor}
\label{cor: regular unimodular triangulation}
Let $s$ be a weakly increasing sequence of positive integers whose first order difference sequence of $s$ is a $0,1$-sequence.
There exists a regular, flag, and unimodular triangulation of $P_n^s$.
\end{cor}
\begin{proof}
By Theorem~\ref{thm: lexicographic grobner basis} we know that the toric ideal of the polytope $A_n^s$ has a quadratic Gr\"obner basis for some term order that has a square-free initial ideal.  
It follows from \cite[Theorem 8.8]{S96} and \cite[Corollary 8.9]{S96} that $A_n^s$ has a regular, flag, and unimodular triangulation whose minimal non-faces are indexed by the lexicographically non-minimal sets of $s$-lecture hall multisets.  
Since $A_n^s$ is unimodularly equivalent to $P_n^s$, we conclude that $P_n^s$ has a regular, flag, and unimodular triangulation.  
\end{proof}

\bigskip

\noindent
{\bf Acknowledgements}. 
Petter Br\"and\'en was partially supported by the Knut and Alice Wallenberg Foundation and Vetenskapsr\aa det.
Liam Solus was partially supported by United States NSF Mathematical Sciences Postdoctoral Research Fellowship (DMS - 1606407).

\bibliography{lecture-hall-bibtex-file}{}
\bibliographystyle{plain}
\end{document}